\documentclass[12pt]{article}

\usepackage[utf8]{inputenc}
\usepackage[T1]{fontenc}
\usepackage{kpfonts}
\usepackage[french,english]{babel}
\usepackage{xspace}
\usepackage{enumerate} %pour modifier les compteurs des enumerations
\usepackage{changepage}

\usepackage[svgnames]{xcolor}
\usepackage[pdftex,leftbars,xcolor]{changebar}

\allowdisplaybreaks[2]

\usepackage[pdftex,a4paper,centering]{geometry}
\usepackage{amsmath,amsthm}
\usepackage{amssymb,latexsym,amsfonts}

\usepackage{textcomp}

\usepackage{pgf}

\usepackage{tikz}
\usetikzlibrary{matrix,arrows}

\usepackage{aliascnt}
\numberwithin{equation}{section}

\usepackage{mathtools}
\usepackage{relsize,exscale}
\usepackage{rotating}

\usepackage{etoolbox}

\usepackage{etex}

\pdfoutput=1
\usepackage[pdftex,%
bookmarks=true,%
bookmarksnumbered=true,%
plainpages=false,%
breaklinks=true,%
colorlinks=false,%
urlcolor=blue,%
citecolor=green,%
linkcolor=red,%
pdftitle={harrco},%
pdfauthor={Olivier Elchinger}]{hyperref}
\usepackage{doi}

\usepackage[hyperpageref]{backref}
\backreffrench
\renewcommand*{\backref}[1]{}  % Disable standard
\renewcommand*{\backrefalt}[4]{% Detailed backref
  \ifcase #1 %
  \relax%(Not cited.)%
  \or
(Cited page~#2.)%
  \else
(Cited pages~#2.)%
  \fi}

\newtheorem{theorem}{Theorem}[section]

\newaliascnt{cor}{theorem}

\aliascntresetthe{cor}

\newaliascnt{prop}{theorem}
\newtheorem{prop}[prop]{Proposition}
\aliascntresetthe{prop}

\newaliascnt{lemma}{theorem}
\newtheorem{lemma}[lemma]{Lemma}
\aliascntresetthe{lemma}

\theoremstyle{definition}
\newaliascnt{defi}{theorem}
\newtheorem{defi}[defi]{Definition}
\aliascntresetthe{defi}

\newaliascnt{example}{theorem}

\aliascntresetthe{example}

\theoremstyle{remark}
\newaliascnt{remark}{theorem}
\newtheorem{remark}[remark]{Remark}
\aliascntresetthe{remark}

\addto\extrasenglish{%
}

\newcommand{\ie}{\textit{i.e.} \/}

\newcommand{\numberset}[1]{\mathbb{#1}}

\newcommand{\nat}{\numberset{N}}

\newcommand{\rational}{\numberset{Q}}

\newcommand{\korps}{\numberset{K}}

\DeclareMathOperator{\Kr}{Ker}
\DeclareMathOperator{\Img}{Im}

\DeclareMathOperator{\Hom}{Hom}

\DeclareMathOperator{\inc}{\textsf{inc}}
\DeclareMathOperator{\pr}{\textsf{pr}}

\DeclareMathOperator{\un}{\mathbf{1}}

\DeclareMathOperator{\Hoch}{Hoch}
\DeclareMathOperator{\Hchains}{C}
\DeclareMathOperator{\Hcochains}{CHoch}

\DeclareMathOperator{\Harr}{Harr}
\DeclareMathOperator{\Harrchains}{Ch}
\DeclareMathOperator{\Harrcochains}{CHarr}

\DeclareMathOperator{\gradcomplex}{C}
\DeclareMathOperator{\gradcohom}{H}

\DeclareMathOperator{\Tens}{\mathcal{T}\!}
\DeclareMathOperator{\Sym}{\mathcal{S}\!}

\newcommand{\phint}{*} %placeholder integer
\newcommand{\varstar}{*}

\title{A formality framework \\ for commutative deformations}
\author{Olivier Elchinger \thanks{The author has been fully supported in the frame of the AFR scheme of the Fonds National
de la Recherche (FNR), Luxembourg with the project QUHACO 8969106}}

\begin{document}

\maketitle

\begin{abstract}
In this article, we use Harrison cohomology to provide a framework for commutative deformations. In particular, Kontsevich's result that formality of (the Hochschild complex of) an associative algebra implies its deformability is adapted for commutative algebras, with the Harrison complex.
\end{abstract}



\section{Introduction}

Kontsevich showed in \cite{K03} the existence of an associative deformation quantization for the general case of smooth Poisson manifolds. He deduced this result from his general ``formality statement''. Endowed with the Gerstenhaber bracket, the continuous Hochschild complex of the algebra $A = \mathcal{C}^\infty(M)$ of smooth functions over a Poisson manifold admits a graded Lie algebraic structure, which controls the deformations of the associative commutative algebra $A$. Kontsevich shows that this complex is linked with its cohomology -- which therefore controls the same deformations -- by a $L_\infty$-quasi-isomorphism, called a formality map.

Considering formality, the case of smooth manifolds is thus rather well understood using continuous Hochschild cohomology, and it is this tool which gives a lot of information about deformability (obstructions, rigidity,\ldots). Moreover, if
the Hochschild complex of an associative algebra is formal in Kontsevich's sense, this algebra admits a quantization by deformation, but the converse does not hold, for example in the case of free algebras, see \cite{El12}.

Since formality methods work well to give complete answers to the deformation quantization in the regular case (both $\mathcal{C}^\infty$ and algebraic) it seems to be interesting -- as proposed by Frønsdal
and Kontsevich in \cite{Fron01,FronKon07} -- to look at the deformation quantization problem for more general singular Poisson manifolds.
The main problem is the fact that the HKR result of a ``simple'' Hochschild cohomology of the algebra of functions, generated by derivations, no longer holds, for example there may be non-trivial 2-cocycles which are symmetric. These symmetric cocycles are infinitesimal \emph{commutative deformations} of the algebra of functions.
In order to systematically investigate commutative associative algebras, Harrison (\cite{Barr68,Harr62}) described combinatorially the ``commutative component'' of the Hochschild complex, and proved that its cohomology is reduced to derivations if and only if the algebra is ``regular''.

\bigskip

The main goal of this work is to adapt the result that formality implies deformation to the case of a commutative algebra, replacing Hochschild complex by Harrison complex.

I am grateful to Prof. Bordemann for his help and useful remarks.

In \autoref{Sec:HochHarr} we recall Hochschild and Harrison (co)homology. \autoref{Sec:bialg} and \autoref{Sec:Euler-idempotents} introduce tools coming from Hopf algebra theory: (co)freeness, convolution products, eulerian idempotents. This gives two descriptions of the Harrison complex, providing a short proof of a result of Barr. Finally, \autoref{Sec:CommDef} presents commutative deformations, and the aforementioned result \autoref{Thm:formality-deformation}. \\

\bigskip

Let $\korps$ be a field containing the rationals.

\section{Hochschild and Harrison (co)homology} \label{Sec:HochHarr}

Let $A$ be a commutative $\korps$-algebra, and consider its Hochschild complex $\Hchains_\phint(A)$ with $\Hchains_n(A) = A \otimes A^{\otimes n}$.

Loday recalls in \cite[4.2]{Lo98} the action of the symmetric group $\mathfrak{S}_n$ on $\Hchains_n(A)$
\begin{align*}
\mathfrak{S}_n \curvearrowright \Hchains_n(A) & \to \Hchains_n(A) \\
\sigma.(a_0,a_1,\dotsc,a_n) & = (a_0,a_{\sigma^{-1}(1)},\dotsc,a_{\sigma^{-1}(n)})
\end{align*}
as well as the shuffle product
\begin{align*}
sh_{p,q} : \Hchains_p(A) \times \Hchains_q(A) & \to \Hchains_{p+q}(A) \\
(a_0,a_1,\dotsc,a_p) \bullet (a'_0,a_{p+1},\dotsc,a_{p+q}) & = \sum_{\sigma \in Sh_{p,q}} \operatorname{sgn}(\sigma) \sigma.(a_0 a'_0,a_1,\dotsc,a_{p+q})
\end{align*}
where $Sh_{p,q}$ are the $(p,q)$-shuffles, elements $\sigma$ of $\mathfrak{S}_{p+q}$ such that $\sigma(1)<\dotsb<\sigma(p)$ and $\sigma(p+1)<\dotsb<\sigma(p+q)$ ; and he also defines the shuffle map
\begin{equation*}
sh = \sum_{\substack{p+q = n \\ p\geqslant 1,q\geqslant 1}}sh_{p,q} : \Hchains_n(A) \to \Hchains_n(A) \quad \text{as the action of the element} \quad sh = \sum_{\substack{\sigma\in Sh_{p,q} \\ p+q = n \\ p\geqslant 1,q\geqslant 1}} \operatorname{sgn}(\sigma) \sigma \in \korps[\mathfrak{S}_n]
\end{equation*}

Endowed with the shuffle product (often noted $\bullet$), the Hochschild complex is a commutative differential graded algebra augmented over $A$. Let $I=\bigoplus_{n>0} \Hchains_n(A)$ be the augmentation ideal. The quotient $\Harrchains(A) =  \Hchains(A)/I^{\bullet 2}$ is a well defined complex since the Hochschild boundary map is a graded derivation for the shuffle product.

For any $A$-module $M$, Hochschild homology and cohomology are given by $\Hoch_\phint(A,M) = H(\Hchains_\phint(A)\otimes_A M)$ and $\Hoch^\phint(A,M) = H(\Hom_A(\Hchains_\phint(A),M))$. The Harrison homology and cohomology are defined as $\Harr_\phint(A,M) = H(\Harrchains_\phint(A)\otimes_A M)$ and $\Harr^\phint(A,M) = H(\Hom_A(\Harrchains_\phint(A),M))$.

Barr already proved in \cite[Theorem 1.1]{Barr68} that there are maps $\Hoch_\phint(A,M) \twoheadrightarrow \Harr_\phint(A,M)$ and $\Harr^\phint(A,M) \hookrightarrow \Hoch^\phint(A,M)$.

\section{Tensorial bialgebras} \label{Sec:bialg}

Let $V$ be a $\korps$-vector space. The tensorial module over $V$ is given by $\Tens V = \bigoplus_{n \in \nat} V^{\otimes n}$.

\subsection{Freeness and cofreeness}

Endowed with the multiplication $\mu$ of concatenation, $(\Tens V,\mu,\un)$ is the free associative algebra over $V$, characterized (up to isomorphism) by the universal property that each morphism $\phi$ from $V$ to an associative algebra $(A,\mu_A)$ factors through $\Tens V$ in $\phi = \inc_V \circ \overline{\phi}$.

Endowed with the comultiplication $\Delta$ of deconcatenation, $(\Tens V,\Delta,\varepsilon)$ is the cofree coassociative conilpotent coalgebra over $V$, characterized (up to isomorphism) by the universal property that each morphism $\phi$ from a coaugmented conilpotent coalgebra $(C,\Delta_C)$ to $V$, \ie satisfying $\phi(1)=0$, factors through $\Tens V$ in $\phi = \pr_V \circ \overline{\phi}$.

Likewise, for any linear map $d : V \to A$, there exists a unique graded derivation along $\overline{\phi}$ noted $\overline{d} : \Tens V \to A$ such that $\overline{d}|_V = d$
%, it satisfies $\overline{d} \circ \mu = \mu_A \circ (\overline{d} \otimes id + id \otimes \overline{d})$ 
; and for any linear map $d : C^+ \to V$ (with $C^+ = \Kr \varepsilon_C$), there exists a unique graded coderivation along $\overline{\phi}$ noted $\overline{d} : C \to \Tens V$ such that $\pr_V \circ \overline{d} = d$.
%, it satisfies $\Delta \circ \overline{d} = (\overline{d} \otimes id + id \otimes \overline{d}) \circ \Delta_C$.
%
%
\begin{align*}
\begin{tikzpicture}[>=angle 90,baseline=(current bounding box.center)]
\matrix (m) [matrix of math nodes,nodes in empty cells,column sep=2em,row sep=2em,text height=1.5ex,text depth=0.25ex,ampersand replacement=\&]
{(\Tens V,\mu) \& \& (A,\mu_A) \\ \& V \\};
\path[->]
(m-1-1) edge node [above] {$\overline{\phi},\overline{d}$} (m-1-3);
\path[right hook->]
(m-2-2) edge (m-1-1);
\path[->]
(m-2-2) edge node [below right] {$\phi,d$} (m-1-3);
\end{tikzpicture}
& &
\begin{tikzpicture}[>=angle 90,baseline=(current bounding box.center)]
\matrix (m) [matrix of math nodes,nodes in empty cells,column sep=2em,row sep=2em,text height=1.5ex,text depth=0.25ex,ampersand replacement=\&]
{(\Tens V,\Delta) \& \& (C,\Delta_C) \\ \& V \\};
\path[->]
(m-1-3) edge node [above] {$\overline{\phi},\overline{d}$} (m-1-1);
\path[->>]
(m-1-1) edge (m-2-2);
\path[->]
(m-1-3) edge node [below right] {$\phi,d$} (m-2-2);
\end{tikzpicture}
\\
\overline{d} \circ \mu = \mu_A \circ (\overline{d} \otimes \overline{\phi} + \overline{\phi} \otimes \overline{d})
& &
\Delta \circ \overline{d} = (\overline{d} \otimes \overline{\phi} + \overline{\phi} \otimes \overline{d}) \circ \Delta_C
\end{align*}

More details on these structures can be found in \cite{LV12}. The emphasis is put on the following formulas using convolutions products. For both the algebra and coalgebra setting the formulas are the same, only the convolution products changes. For more detailed proofs, see \cite{El12}.

The algebra morphism $\overline{\phi}$ induced by $\phi : V \to A$ is computed as $\overline{\phi} = \sum_{n \in \nat} \phi^{\star n}$, the geometric serie using the convolution product $\star$ with respect to the multiplication $\mu_A$ and the comultiplication of deconcatenation $\Delta$. The derivation $\overline{d}$ along $\overline{\phi}$ induced by $d$ and $\phi$ can be computed as $\overline{d} = \overline{\phi} \star d \star \overline{\phi}$.

The coalgebra morphism $\overline{\phi}$ coinduced by $\phi : C^+ \to V$ is computed as $\overline{\phi} = \sum_{n \in \nat} \phi^{\star n}$, the geometric serie using the convolution product $\star$ with respect to the multiplication of concatenation $\mu$ and the comultiplication $\Delta_C$. The coderivation $\overline{d}$ along $\overline{\phi}$ coinduced by $d$ and $\phi$ can be computed as $\overline{d} = \overline{\phi} \star d \star \overline{\phi}$. \\

Moreover, $(\Tens V,\mu,\Delta_{sh},\un,\varepsilon)$ is a bialgebra, $\Delta_{sh}$ being the morphism of associative algebras $\Delta_{sh} : (\Tens V,\mu) \to (\Tens V \otimes \Tens V,\mu^{[2]})$ induced by $\inc_V \otimes \un + \un \otimes \inc_V$.

Also, $(\Tens V,\mu_{sh},\Delta,\un,\varepsilon)$ is a bialgebra, $\mu_{sh}$ being the morphism of coassociative coalgebras $\mu_{sh} : (\Tens V \otimes \Tens V,\Delta^{[2]}) \to (\Tens V,\Delta)$ coinduced by $\pr_V \otimes \varepsilon + \varepsilon \otimes \pr_V$.

%Furthermore, $\Tens V$ is even a Hopf algebra in these two different ways, the antipode is given by $S(x) = -\tilde{x}$ for $x \in V$.

The shuffle product can also be seen as the commutative product resulting on the quotient $\Sym V = \Tens V/(x \otimes y - (-1)^{|x| |y|} y \otimes x,\quad x,y \in \Tens V)$. Since $\Delta_{sh}$ is cocommutative, it factors through the quotient, and thus $(\Sym V,\mu_{sh},\Delta_{sh},\un,\varepsilon)$ is also a bialgebra.

\subsection{Toolbox on operations}

In this section, we present some relations between product, composition, convolution and counit which will be used later. Let $\alpha : V \to V$ (or $\alpha_i$) be a linear map.

Since $V \otimes \korps \cong V$, we have that $\mu_{sh} \circ (\alpha \otimes \un \varepsilon) = \alpha \otimes \varepsilon$ are the same map from $V \otimes V \to V$ since they send $a \otimes b \mapsto \alpha(a)\varepsilon(b)$.

Let $\phi : \Tens V \otimes \Tens V \to \Tens V$ be a coalgebra morphism, meaning that $\Delta \circ \phi = (\phi \otimes \phi) \circ \Delta^{[2]}$, with $\Delta^{[2]} = (id \otimes \tau \otimes id) \circ (\Delta \otimes \Delta)$, that is $\Delta^{[2]} = \operatorname{perm} \circ (\Delta \otimes \Delta)$, the coproduct on each factor followed by a permutation so that the morphism is applied to the right elements.
We have
\begin{align*}
(\alpha_1 \star \alpha_2) \circ \phi & {}= \mu_{sh} \circ (\alpha_1 \otimes \alpha_2) \circ \Delta \circ \phi \\
& {}= \mu_{sh} \circ (\alpha_1 \otimes \alpha_2) \circ (\phi \otimes \phi) \circ \Delta^{[2]} \\
& {}= \mu_{sh} \circ (\alpha_1 \circ \phi \otimes \alpha_2 \circ \phi) \circ \Delta^{[2]} = (\alpha_1 \circ \phi) \star_2 (\alpha_2 \circ \phi),
\end{align*}
with the convolution product $\_ \star_2 \_ = \mu_{sh} \circ (\_ \otimes \_)\circ \Delta^{[2]}$.
The property of coalgebra morphism also reads $\Delta^{(n-1)} \circ \phi = \phi^{\otimes n} \circ \Delta^{[n]}$, where $\Delta^{(n-1)} : \Tens V \to \Tens V^{\otimes n}$ is the $(n-1)$-fold of the associative coproduct, and $\Delta^{[n]} = \operatorname{perm} \circ (\Delta^{(n-1)} \otimes \Delta^{(n-1)})$. We have
\begin{equation*}
(\alpha_1 \star \dotsb \star \alpha_n) \circ \phi = (\alpha_1 \circ \phi) \star_n \dotsb \star_n (\alpha_n \otimes \phi) \qquad \text{with}\quad
\begin{aligned}
\_\star\_ \dotsb \_\star \_ = \mu_{sh} \circ (\_\otimes\_ \dotsb \_\otimes\_) \circ \Delta^{(n-1)} \\
\_\star_n\_ \dotsb \_\star_n\_ = \mu_{sh} \circ (\_\otimes\_ \dotsb \_\otimes\_) \circ \Delta^{[n]}
\end{aligned}
\end{equation*}
and we will write collectively $\varstar$ for those second kind of convolutions. Taking $\alpha_i = \alpha$ and summing the previous equalities gives
\begin{equation*}
e^{\star \alpha} \circ \phi = e^{\varstar(\alpha \circ \phi)}.
\end{equation*}

Using this with $\phi = id \otimes \varepsilon$, which indeed is a coalgebra morphism, we obtain
\begin{gather*}
(\alpha_1 \star \alpha_2) \otimes \varepsilon = (\alpha_1 \star \alpha_2) \circ \phi = (\alpha_1 \circ \phi) \varstar (\alpha_2 \circ \phi) = (\alpha_1 \otimes \varepsilon) \varstar (\alpha_2 \otimes \varepsilon)
\intertext{and thus}
e^{\star \alpha} \otimes \varepsilon = e^{\varstar(\alpha \otimes \varepsilon)}.
\end{gather*}

\section{Eulerian idempotents} \label{Sec:Euler-idempotents}

Following Loday and Vallette \cite[4.5]{Lo98} and \cite[1.3.11]{LV12}, we define the eulerian idempotents on the commutative Hopf algebra $(\Tens V,\mu_{sh},\Delta,\un,\varepsilon)$. We consider its convolution algebra $(\Hom(\Tens V,\Tens V),\star,\un \varepsilon)$, where the convolution product is $\_ \star \_ = \mu_{sh} \circ (\_ \otimes \_)\circ \Delta$. We write $id = \un \varepsilon + J$ so that $J$ is the identity on $V^{\otimes n}$ except for $n = 0$ on which it is $0$. We define
\begin{equation*}
e^{(1)} \coloneqq \log^\star(id) = \log^\star(\un \varepsilon + J) = \sum_{n \geqslant 1} (-1)^{n+1} \frac{J^{\star n}}{n}
\end{equation*}

In weight $n$ we get that $e^{(1)} : V^{\otimes n} \to V^{\otimes n}$ is given by $e^{(1)}(x_1 \dotsm x_n) = e_n^{(1)} \cdot (x_1 \dotsm x_n)$ for some uniquely defined element $e_n^{(1)} \in \rational[\mathfrak{S}_n]$. These elements are called the first \emph{eulerian idempotents}. For $i \geqslant 1$, we define
\begin{equation*}
e^{(i)} \coloneqq \frac{(e^{(1)})^{\star i}}{i!}
\end{equation*}

Loday shows \cite[Proposition 4.5.3]{Lo98} that the elements $e_n^{(i)} \in \rational[\mathfrak{S}_n]$ are orthogonal idempotents.

In low dimensions, the eulerian idempotents are
\begin{align*}
n=1 \qquad & e_1^{(1)} = id; \\
n=2 \qquad & e_2^{(1)} = \frac{1}{2} \big(id + (1 2)\big), \quad e_2^{(2)} = \frac{1}{2} \big(id - (1 2)\big); \\
n=3 \qquad & e_3^{(1)} = \frac{1}{3}id - \frac{1}{6}\big((1 2 3)+(1 3 2)-(1 2)-(2 3)\big) - \frac{1}{3}(1 3), \\
& e_3^{(2)} = \frac{1}{2} \big(id+(1 3)\big), \quad e_3^{(3)} = \frac{1}{6}\big(id+(1 2 3)+(1 3 2)-(1 2)-(2 3)-(1 3)\big).
\end{align*}

\begin{prop}
The first eulerian idempotent $e^{(1)}$ is a derivation for $\mu_{sh}$ along $\un \varepsilon$, \ie
\begin{equation*}
e^{(1)} \circ \mu_{sh} = \mu_{sh} \circ (e^{(1)} \otimes \un \varepsilon + \un \varepsilon \otimes e^{(1)}).
\end{equation*}
\end{prop}

\begin{proof}
Writing $\mu_{sh} = id \circ \mu_{sh} = e^{\star e^{(1)}} \circ \mu_{sh} = e^{\varstar(e^{(1)} \circ \mu_{sh})}$, we will show that $e^{\varstar(e^{(1)} \otimes \varepsilon + \varepsilon \otimes e^{(1)})} = e^{\varstar(e^{(1)} \circ \mu_{sh})}$, which gives the result since $e^{(1)} \otimes \varepsilon + \varepsilon \otimes e^{(1)} = \mu_{sh} \circ (e^{(1)} \otimes \un \varepsilon + \un \varepsilon \otimes e^{(1)})$.

We have $e^{\varstar(A+B)} = e^{\varstar A}\varstar e^{\varstar B}$, providing $A \varstar B = B \varstar A$. Set $A =\alpha \otimes \varepsilon$ and $B = \varepsilon \otimes \alpha$, with $\alpha$ a linear map. Let $a,b \in \Tens V$, we use Sweedler notation $\Delta(a) = \sum_{(a)} a^{(1)} \otimes a^{(2)}$. Note that $|a| = |a^{(1)}| + |a^{(2)}|$ for any elements $a^{(1)},a^{(2)}$ of the sum. We have
\begin{align*}
\big( \alpha \otimes \varepsilon \varstar \varepsilon \otimes \alpha \big)(a \otimes b) & {}= \mu_{sh} \circ (\alpha \otimes \varepsilon \otimes \varepsilon \otimes \alpha) \circ \Delta^{[2]}(a \otimes b) \\
& {}= \mu_{sh} \left(\sum_{(a),(b)} (-1)^{|a^{(2)}| |b^{(1)}|} \alpha(a^{(1)}) \varepsilon(b^{(1)}) \otimes \varepsilon(a^{(2)}) \alpha(b^{(2)}) \right) \\
& {}= \mu_{sh} \left(\sum_{(a),(b)} \alpha\big(a^{(1)} \varepsilon(a^{(2)})\big) \otimes \alpha\big(b^{(2)} \varepsilon(b^{(1)})\big) \right) \\
& = \alpha(a) \bullet \alpha(b)
\end{align*}
since terms with elements $a^{(2)}$ or $b^{(1)}$ of degree different from zero are killed by the counity; and since $|a^{(2)}| = |a| - |a^{(1)}|$, $|b^{(1)}| = |b| - |b^{(2)}|$,
\begin{align*}
\big( \varepsilon \otimes \alpha \varstar \alpha \otimes \varepsilon \big)(a \otimes b) & {}= \mu_{sh} \circ (\varepsilon \otimes \alpha \otimes \alpha \otimes \varepsilon) \circ \Delta^{[2]}(a \otimes b) \\
& {}= \mu_{sh} \left(\sum_{(a),(b)} (-1)^{|a^{(2)}| |b^{(1)}|} \varepsilon(a^{(1)}) \alpha(b^{(1)}) \otimes \alpha(a^{(2)}) \varepsilon(b^{(2)}) \right) \\
& {}= (-1)^{|a| |b|} \alpha(b) \bullet \alpha(a) = \alpha(a) \bullet \alpha(b)
\end{align*}
Taking $\alpha = e^{(1)}$, we thus have
\begin{align*}
e^{\varstar(e^{(1)} \otimes \varepsilon + \varepsilon \otimes e^{(1)})} & {}= e^{\varstar(e^{(1)} \otimes \varepsilon)} \varstar e^{\varstar(\varepsilon \otimes e^{(1)})} = \left( e^{\star e^{(1)}} \otimes \varepsilon \right) \varstar \left( \varepsilon \otimes e^{\star e^{(1)}} \right) \\
& {}= (id \otimes \varepsilon) \varstar (\varepsilon \otimes id) = \mu_{sh} \circ (id \otimes \varepsilon \otimes \varepsilon \otimes id) \circ \Delta^{[2]} = \mu_{sh} \circ (id \otimes id) = \mu_{sh} = e^{\varstar(e^{(1)} \circ \mu_{sh})}.
\end{align*}
\end{proof}

This proposition implies the equivalence of the two original definitions of Harrison (co)homology of a commutative algebra $A$ as given by Harrison \cite{Harr62} and Barr \cite{Barr68}.

\begin{theorem}
The complexes $\Harrchains(A) = \Hchains(A)/I^{\bullet 2}$ and $e^{(1)} \Hchains(A)$ are isomorphic.
\end{theorem}

\begin{proof}
Let $a,b \in \Hchains(A)$. Since $e^{(1)}$ is a derivation of $\mu_{sh}$ along $\un \varepsilon$, we have
\begin{equation*}
e^{(1)}(a \bullet b) = e^{(1)}(a) \varepsilon(b) + \varepsilon(a) e^{(1)}(b).
\end{equation*}
If $a,b \in I$ are elements in the ideal of augmentation, then $\varepsilon(a) = 0 = \varepsilon(b)$, thus $e^{(1)}(a \bullet b) = 0$, hence $I \bullet I \subset \Kr(e^{(1)})$. Also (Sweedler's summations implied)
\begin{equation*}
e^{(1)}(a) = \sum_{n \geqslant 1} (-1)^{n+1} \frac{id^\star}{n}(a) = a-\frac{1}{2} a^{(1)} \bullet a^{(2)} + \frac{1}{3} a^{(1)} \bullet a^{(2)} \bullet a^{(3)} - \dotsb
\end{equation*}
so $a-e^{(1)}(a) \in I \bullet I$, hence $\Img(id - e^{(1)}) = \Kr(e^{(1)}) \subset I \bullet I$.

So we have the decomposition
\[
\Hchains(A) = (id-e^{(1)})(A) \oplus e^{(1)}(A) = I^{\bullet 2} \oplus e^{(1)}(A)
\]
which gives the result.
\end{proof}

Barr's proof of \cite[Proposition 2.5]{Barr68} consists in a construction by induction of a sequence $e_n \in \korps[\mathfrak{S}_n]$ of idempotent maps commuting with the Hochschild boundary map and leaving the shuffle products $sh_{p,q}$ invariant. Here this proof use the property of the whole map $e^{(1)}$ of being a graded derivation. Note that $e^{(1)}_n = id - e_n$ in Barr's notation.

In particular, this gives two descriptions of Harrison cochains:
\[
\Harrcochains(A,M) = \Hom(\Hchains(A)/I^{\bullet 2},M) = \Hom(e^{(1)}(A),M)
\]
they can be viewed as maps $\Tens A \to A$ that cancel on shuffles, or invariant by the first eulerian idempotent.

\begin{remark}
The second Harrison module $\Harrcochains^2(A,M) = \{f : A^{\otimes 2} \to M, f(a,b)=f(b,a)\}$ consists of symmetric maps.
\end{remark}

\section{Commutative deformations} \label{Sec:CommDef}

Let $(A,\mu_0)$ be a commutative $\korps$-algebra. In \cite{FronKon07}, Fr{\o{}}nsdal defines commutative deformations. A formal, abelian $*$-product on $A$ is a commutative, associative product on the space $A[[\lambda]]$ of formal power series in the formal parameter $\lambda$ with coefficients in $A$, given by formal series
\begin{equation*}
f * g = \sum_{n \in \nat} \lambda^n \mu_n(f,g).
\end{equation*}

Associativity for $*$ is the condition $(f * g) * h = f * (g * h)$ or equivalently $A_n(f,g,h) = 0$ for all $n \in \nat$, where $A_n$ is the associator of order $n$ for $*$
\begin{equation*}
A_n(f,g,h) \coloneqq \sum_{k=0}^n \big(\mu_k(\mu_{n-k}(f,g),h) - \mu_k(f,\mu_{n-k}(g,h))\big).
\end{equation*}

For any product $\mu$, its associator $A(a,b,c) \coloneqq \mu(\mu(a,b),c)-\mu(a,\mu(b,c))$ satisfies
\begin{equation} \label{associator-identity}
0 = [\mu,A]_G(a,b,c,d) = \mu(A(a,b,c),d) + \mu(a,A(b,c,d)) - A(\mu(a,b),c,d) + A(a,\mu(b,c),d) - A(a,b,\mu(c,d))
\end{equation}
with $[~,~]_G$ the Gerstenhaber bracket. %, whose definition is recalled below. ?
\\

Let $A = \sum_{n \in \nat} \lambda^n A_n$ be the associator of a $*$-product $*$. Suppose that $*$ is associative to order $r \geqslant 1$, \ie $A_0 = \dotsb = A_r = 0$.
Equation \eqref{associator-identity} at order $\lambda^{r+1}$ reads $0 = (\delta A_{r+1})(a,b,c,d)$, with $\delta$ the Hochschild coboundary, hence $A_{r+1}$ is a Hochschild $3$-cocycle. Moreover $A_{r+1} = -\delta \mu_{r+1} + A_{r+1}'$, where $A_{r+1}'$ is $A_{r+1}$ without the first and last term in the sum. This shows that $A_{r+1}'$ is also a $3$-cocycle, and $A_{r+1} = 0 \Leftrightarrow A_{r+1}' = \delta \mu_{r+1}$, so that $*$ is associative to order $r+1$ is equivalent to $A_{r+1}'$ being a $3$-coboundary.

This proves that the obstruction to promote associativity from order $r$ to order $r+1$ are in $\Hoch^3(A,A)$. Moreover, if $\mu_i,\ 1 \leqslant i \leqslant r$ are symmetric, then a direct computation shows that $A_{r+1}'$ is invariant by $e^{(1)}_3$, so the obstructions to extend a formal abelian $*$-product to higher orders are more precisely in $\Harr^3(A,A)$. \\

Barr showed that Harrison cohomology is included in Hochschild cohomology, but it is already the case for the complexes as differential graded Lie algebras.

\begin{prop} \label{Prop:sub-dgL}
The Harrison complex of cochains $(\Harrcochains(A,A)[1],\delta,[~,~]_G)$ is a differential graded sub-Lie algebra of the Hochschild complex $(\Hcochains(A,A)[1],\delta,[~,~]_G)$.
\end{prop}

We first prove the following lemma, see also \cite[A.1]{BGHHW05}.

\begin{lemma}
Cochains of $\Harrcochains(A,A)$ induces derivations of $(\Tens A,\mu_{sh})$.
\end{lemma}

\begin{proof}
Let $d : A^{\otimes k} \to A$ be a cochain in $\Harrcochains^k(A,A) \subset \Hcochains^k(A,A)$. It induces $\overline{d}$, coderivation of $(\Tens A,\Delta)$. We want to show that it is also a derivation for $\mu_{sh}$, \ie
\begin{equation*}
\overline{d} \circ \mu_{sh} = \mu_{sh} \circ (\overline{d} \otimes id + id \otimes \overline{d}).
\end{equation*}
Since $\mu_{sh} : (\Tens A \otimes \Tens A, \Delta^{[2]}) \to (\Tens A,\Delta)$ is a coalgebra morphism, both sides of the equation are coderivation from $(\Tens A \otimes \Tens A,\Delta^{[2]})$ to $(\Tens A,\Delta)$ along $\mu_{sh}$.

Projecting on $A$, we have on the left-hand side $d \circ \mu_{sh}(a \otimes b) = d(a \bullet b)$ and on the right-hand side $(\pr_A \otimes \varepsilon + \varepsilon \otimes \pr_A) \circ (\overline{d} \otimes id + id \otimes \overline{d})(a \otimes b) = (d \otimes \varepsilon + \varepsilon \otimes d)(a \otimes b) = d(a) \varepsilon(b) + \varepsilon(a) d(b)$ because $\varepsilon \circ \overline{d} = 0$. But since $d$ vanishes on $I^{\bullet 2}$ and $\varepsilon$ on $I$, the two expressions are equal for all $a,b \in \Tens A$. Since the left and right-hand side are coderivations along $\mu_{sh}$ having the same projection, they must be equal by unicity.
\end{proof}

\begin{proof}[Proof of \autoref{Prop:sub-dgL}]
Let $f,g \in \Harrcochains(A,A)[1]$. Using the previous lemma, we have
\begin{align*}
[f,g]_G \circ \mu_{sh} & {}= f \circ \overline{g} \circ \mu_{sh} + (-1)^{|f| |g|} g \circ \overline{f} \circ \mu_{sh} \\
& {}= f \circ \mu_{sh} \circ (\overline{g} \otimes id + id \otimes \overline{g}) + (-1)^{|f| |g|} g \circ \mu_{sh} \circ (\overline{f} \otimes id + id \otimes \overline{f})
\intertext{thus}
[f,g]_G(a \bullet b) & {}= f \big(\overline{g}(a) \bullet b \pm a \bullet \overline{g}(b) \big) + (-1)^{|f| |g|} g \big( \overline{f}(a) \bullet b \pm a \bullet \overline{f}(b) \big) = 0
\end{align*}
hence the vanishing of $f$ and $g$ on $I^{\bullet 2}$ imply the one of $[f,g]_G$ on $I^{\bullet 2}$, so $\Harrcochains(A,A)[1]$ is closed for the Gerstenhaber bracket.
\end{proof}

We recall Kontsevich's notion of formality. For better readability, we note $\gradcomplex = \Hcochains(A,A)$ and $\gradcohom = \Hoch(A,A)$ the Hochschild complex and cohomology of $A$.

\begin{defi} 
The complex $\gradcomplex$ is called \emph{formal} if there is a $L_\infty$-quasi-isomorphism $\Phi : \Sym(\gradcohom[2]) \to \Sym(\gradcomplex[2])$ (morphism of differential graded coalgebras of degree $0$), \ie
\begin{equation} \label{L_inf_morphism}
(\Phi \otimes \Phi) \circ \Delta_{\Sym{} \gradcohom[2]} = \Delta_{\Sym{} \gradcomplex[2]} \circ \Phi \qquad \textrm{and} \qquad \overline{b+D} \circ \Phi = \Phi \circ \overline{d},
\end{equation}
such that the restriction $\Phi_1$ of $\Phi$ to $H[2]$ is a section. The map $\Phi$ is called a formality map.
\end{defi}

Here $b \coloneqq [\mu_0,~]_G$ is the same as the Hochschild coboundary $\delta$ up to a global sign. We recall that the projection of the Gerstenhaber bracket gives a graded Lie bracket $[~,~]_s$ on the shifted cohomology space $\gradcohom[1]$. The maps $D \coloneqq [~,~]_G[1]$ and $d \coloneqq [~,~]_s[1]$ denote the shifted brackets, which are symmetric; $\overline{D}$ and $\overline{d}$ are the induced coderivations on $\Sym(\gradcomplex[2])$ and $\Sym(\gradcohom[2])$.

By extension, we will say that an associative algebra $A$ is \emph{formal} if it is the case for its Hochschild complex $\gradcomplex$. For a commutative algebra $A$, we keep the same definition of formality, but now taking $\gradcomplex = \Harrcochains(A,A)$ and $\gradcohom = \Harr(A,A)$ the Harrison complex and cohomology of $A$.

\begin{theorem} \label{Thm:formality-deformation}
(commutative) formality $\Rightarrow$ (commutative) déformation
\end{theorem}

\begin{proof}
For associative algebras, the result goes back to Kontsevich \cite{K03}, with the given framework, it adapts well to commutative algebras. We follow here the presentation of \cite[pp 321--322]{BM08}.
Let $\pi\in \gradcohom^2[[\lambda]] = \gradcohom[2]^0[[\lambda]]$. We want to construct a formal associative (commutative) deformation $\mu=\mu_0+\mu_*$ where $\mu_* \coloneqq \sum_{r=1}^{\infty} \lambda^r \mu_r$ such that the cohomology class $[\mu_1]$ of $\mu_1$ is equal to $\pi$. A necessary condition for this is
\begin{equation*}
[\pi,\pi]_s = 0
\end{equation*}
so we suppose the chosen element $\pi$ satisfies it.

Consider $\Sym(\gradcohom[2])[[\lambda]]$ and $\Sym(\gradcomplex[2])[[\lambda]]$ as topological bialgebras (with respect to the $\lambda$-adic topology) with the canonical extension of all the structure maps. Note that the tensor product is no longer algebraic, but given by $\big(\Sym(\gradcohom[2]) \otimes \Sym(\gradcohom[2])\big)[[\lambda]]$. For a general graded vector space $V$ it can be easily seen that the group-like elements of $\Sym{V}[[\lambda]]$ are no longer exclusively given by $\un$, but by exponential functions of any primitive elements of degree zero, \ie they take the form $e^{\bullet \lambda v}$ with $v\in V^0[[\lambda]]$. The image $\Phi(e^{\bullet \lambda\pi})$ of the grouplike element $e^{\bullet \lambda \pi}$ in $\Sym(\gradcohom[2])[[\lambda]]$ under the formality map $\Phi$ is a grouplike element in $\Sym(\gradcomplex[2])[[\lambda]]$ and thus takes the form $e^{\bullet \mu_*}$ with $\mu_* \in \lambda \gradcomplex^2[[\lambda]]$. Since $[\pi,\pi]_s = 0$ it follows that $d(e^{\bullet \lambda\pi})=0$, and therefore $(b+D)(e^{\bullet \lambda\mu_*}) = 0$. Projecting this last equation to $\gradcomplex[2]^0[[\lambda]] = \gradcomplex^2[[\lambda]]$, we get the \emph{Maurer-Cartan Equation}
\begin{equation*}
0 = b\mu_* + \frac{1}{2} [\mu_*,\mu_*]_G = \frac{1}{2} [\mu_0 + \mu_*,\mu_0 + \mu_*]_G,
\end{equation*}
showing the associativity of $\mu = \mu_0 + \mu_*$. Hence $\mu \coloneqq \mu_0 + \mu_*$ is a formal associative deformation of the algebra $(A,\mu_0)$.
In the commutative case with $\gradcomplex = \Harrcochains(A,A)$, $\mu_* \in \lambda \gradcomplex^2[[\lambda]]$ is equivalent to the commutativity of $\mu_i$ for $i\geqslant 1$, so the resulting product $\mu$ is commutative.
\end{proof}

\bibliographystyle{amsalpha}
\bibliography{biblio-harrco}
\nocite{*}

\end{document}